\documentclass{amsart}
\usepackage{amssymb}
\usepackage{amsmath}
\usepackage{amsfonts}

\newtheorem{theorem}{Theorem}
\theoremstyle{plain}

\newtheorem{corollary}{Corollary}

\newtheorem{remark}{Remark}

\numberwithin{equation}{section}
\input{tcilatex}

\begin{document}
\title[]{on the $(1,1)-$tensor bundle \\
with Cheeger-Gromoll type metric}
\author{Aydin GEZER}
\address{Ataturk University, Faculty of Science, Department of Mathematics,
25240, Erzurum-Turkey.}
\email{agezer@atauni.edu.tr}
\author{Murat ALTUNBAS}
\address{Erzincan University, Faculty of Science and Art, Department of
Mathematics, 24030, Erzincan-Turkey.}
\email{maltunbas@erzincan.edu.tr}

\begin{abstract}
The main purpose of the present paper is to construct Riemannian almost
product structures on the $(1,1)-$tensor bundle equipped with
Cheeger-Gromoll type metric over a Riemannian manifold and present some
results concerning with these structures.
\end{abstract}

\maketitle

\section{\protect\bigskip \textbf{Introduction}}

\noindent A Riemannian metric on tangent bundle of a Riemannian manifold had
been defined by E. Musso and F. Tricerri \cite{Musso} who, inspired by the
paper \cite{Cheeger} of J. Cheeger and D. Gromoll, called it the
Cheeger-Gromoll metric. The metric was defined by J. Cheeger and D. Gromoll;
yet, there were E. Musso and F. Tricerri who wrote down its expression,
constructed it in a more \textquotedblright
comprehensible\textquotedblright\ way, and gave it the name. The Levi-Civita
connection of the Cheeger-Gromoll metric and its Riemannian curvature tensor
are calculated by M. Sekizawa in \cite{Sekizawa} (for more details see \cite%
{Gudmundsson}). In \cite{Peyghan}, the metric was considered on $(1,1)$%
-tensor bundle by E. Peyghan and his collaborates and examined its some
geometric properties. In this paper, our aim is to study some almost product
structures on the $(1,1)-$tensor bundle endowed with the Cheeger-Gromoll
type metric. We get sequently conditions for the $(1,1)-$tensor bundle
endowed with the Cheeger-Gromoll type metric and some almost product
structures to be locally decomposable Riemannian manifold and Riemannian
almost product $\mathcal{W}_{3}-$manifold. Finally, we consider a product
conjugate connection on the $(1,1)-$tensor bundle with the Cheeger-Gromoll
type metric and give some results related to the connection.

Throughout this paper, all manifolds, tensor fields and connections are
always assumed to be differentiable of class $C^{\infty }$. Also, we denote
by $\Im _{q}^{p}(M)$ the set of all tensor fields of type $(p,q)$ on $M$,
and by $\Im _{q}^{p}(T_{1}^{1}(M))$ the corresponding set on the $(1,1)-$%
tensor bundle $T_{1}^{1}(M)$.

\section{Preliminaries}

Let $M$ be a differentiable manifold of class $C^{\infty }$ and finite
dimension $n$. Then the set $T_{1}^{1}(M)=\underset{P\in M}{\cup }%
T_{1}^{1}(P)$ is, by definition, the tensor bundle of type $(1,1)$ over $M$,
where $\cup $ denotes the disjoint union of the tensor spaces $T_{1}^{1}(P)$
for all $P\in M$. For any point $\tilde{P}$ of $T_{1}^{1}(M)$ such that $%
\tilde{P}\in T_{1}^{1}(M)$, the surjective correspondence $\tilde{P}%
\rightarrow P$ determines the natural projection $\pi
:T_{1}^{1}(M)\rightarrow M$. The projection $\pi $ defines the natural
differentiable manifold structure of $T_{1}^{1}(M)$, that is, $T_{1}^{1}(M)$
is a $C^{\infty }$-manifold of dimension $n+n^{2}$. If $x^{j}$ are local
coordinates in a neighborhood $U$ of $P\in M$, then a tensor $t$ at $P$
which is an element of $T_{1}^{1}(M)$ is expressible in the form $%
(x^{j},t_{j}^{i})$, where $t_{j}^{i}$ are components of $t$ with respect to
the natural base. We may consider $(x^{j},t_{j}^{i})=(x^{j},x^{\bar{j}%
})=(x^{J})$, $j=1,...,n$, $\bar{j}=n+1,...,n+n^{2}$, $J=1,...,n+n^{2}$ as
local coordinates in a neighborhood $\pi ^{-1}(U)$.

Let $X=X^{i}\frac{\partial }{\partial x^{i}}$ and $A=A_{j}^{i}\frac{\partial 
}{\partial x^{i}}\otimes dx^{j}$ be the local expressions in $U$ of a vector
field $X$ \ and a $(1,1)$ tensor field $A$ on $M$, respectively. Then the
vertical lift $^{V}A$ of $A$ and the horizontal lift $^{H}X$ of $X$ are
given, with respect to the induced coordinates, by\noindent 
\begin{equation}
{}^{V}A=\left( 
\begin{array}{l}
{{}^{V}A^{j}} \\ 
{{}^{V}A^{\overline{j}}}%
\end{array}%
\right) =\left( 
\begin{array}{c}
{0} \\ 
{A_{j}^{i}}%
\end{array}%
\right) ,  \label{A2.1}
\end{equation}%
and

\begin{equation}
^{H}X=\left( 
\begin{array}{c}
{{}^{H}X^{j}} \\ 
{{}^{H}X^{\overline{j}}}%
\end{array}%
\right) =\left( 
\begin{array}{c}
X{^{j}} \\ 
X{^{s}(\Gamma _{sj}^{m}t_{m}^{i}-\Gamma _{sm}^{i}t_{j}^{m})}%
\end{array}%
\right) ,  \label{A2.2}
\end{equation}%
where $\Gamma _{ij}^{h}$ are the coefficients of the connection $\nabla $ on 
$M$.

Let $\varphi \in \Im _{1}^{1}(M)$. The global vector fields $\gamma \varphi $
and $\tilde{\gamma}\varphi \in \Im _{0}^{1}(T_{1}^{1}(M))$ are respectively
defined by%
\begin{equation*}
{\gamma \varphi =\left( 
\begin{array}{c}
{0} \\ 
{t_{j}^{m}\varphi _{m}^{i}}%
\end{array}%
\right) ,\tilde{\gamma}\varphi =\left( 
\begin{array}{c}
{0} \\ 
{(t_{m}^{i}\varphi _{j}^{m})}%
\end{array}%
\right) }
\end{equation*}%
with respect to the coordinates $(x^{j},x^{\bar{j}})$ in $T_{1}^{1}(M)$,
where $\varphi _{j}^{i}$ are the components of $\varphi $.

The Lie bracket operation of vertical and horizontal vector fields on $%
T_{1}^{1}(M)$ is given by the formulas%
\begin{equation}
\left\{ 
\begin{array}{l}
{\left[ {}^{H}X,{}^{H}Y\right] ={}^{H}\left[ X,Y\right] +(\tilde{\gamma}%
-\gamma )R(X},{Y}), \\ 
{\left[ {}^{H}X,{}^{V}A\right] ={}^{V}(\nabla _{X}A),} \\ 
{\left[ {}^{V}A,{}^{V}B\right] =0}%
\end{array}%
\right.  \label{A2.3}
\end{equation}%
for any $X,$ $Y$ $\in \Im _{0}^{1}(M)$ and $A$, $B\in \Im _{1}^{1}(M)$,
where $R$ is the curvature tensor field of the connection $\nabla $ on $M$
defined by $R\left( X,Y\right) =\left[ \nabla _{X},\nabla _{Y}\right]
-\nabla _{\left[ X,Y\right] }$ \ and ${(\tilde{\gamma}-\gamma )R(X},{Y}%
)=\left( 
\begin{array}{c}
0 \\ 
{t_{m}^{i}R}_{klj}^{\text{ \ \ \ }m}X^{k}Y^{l}{-t_{j}^{m}R}_{klm}^{\text{ \
\ \ }i}X^{k}Y^{l}%
\end{array}%
\right) $ (for details, see \cite{Cengiz,Salimov3}).

\section{\protect\bigskip Riemannian almost product structures on the $%
(1,1)- $tensor bundle with Cheeger-Gromoll type metric}

An $n-$dimensional manifold $M$ in which a $(1,1)$ tensor field $\varphi $
satisfying $\varphi ^{2}=id$, $\varphi \neq \pm id$ is given is called an
almost product manifold. \ A Riemannian almost product manifold $(M,\varphi
,g)$ is a manifold $M$ with an almost product structure $\varphi $ and a
Riemannian metric $g$ such that%
\begin{equation}
g(\varphi X,Y)=g(X,\varphi Y)  \label{B3.0}
\end{equation}%
for all $X,Y\in \Im _{0}^{1}(M)$. Also, the condition (\ref{B3.0}) is
referred as purity condition for $g$ with respect to $\varphi $. The almost
product structure $\varphi $ is integrable, i.e. the Nijenhuis tensor $%
N_{\varphi }$, determined by%
\begin{equation*}
N_{\varphi }(X,Y)=\left[ \varphi X,\varphi Y\right] -\varphi \left[ \varphi
X,Y\right] -\varphi \left[ X,\varphi Y\right] +\left[ X,Y\right] 
\end{equation*}%
for all $X,Y\in \Im _{0}^{1}(M)$ is zero then the Riemannian almost product
manifold $(M,\varphi ,g)$ is called a Riemannian product manifold. A locally
decomposable Riemannian manifold can be defined as a triple $(M,\varphi ,g)$
which consists of a smooth manifold $M$ endowed with an almost product
structure $\varphi $ and a pure metric $g$ such that $\nabla \varphi =0$,
where $\nabla $ is the Levi-Civita connection of $g$. It is well known that
the condition $\nabla \varphi =0$ is equivalent to decomposability of the
pure metric $g$ \cite{SalimovAkbulut}, i.e. $\Phi _{\varphi }g=0$, where $%
\Phi _{\varphi }$ is the Tachibana operator \cite{Tachibana,YanoAko}: $(\Phi
_{\varphi }g)(X,Y,Z)=(\varphi X)({}g(Y,Z))-X(g(\varphi Y,Z))+g((L_{Y}\varphi
)X,Z)+g(Y,(L_{Z}$ $\varphi )X)\,$.

Let $T_{1}^{1}(M)$ be the $(1,1)-$tensor bundle over a Riemannian manifold $%
(M,g)$. For each $P\in M$, the extension of scalar product $g$ (marked by $G$%
) is defined on the tensor space $\pi ^{-1}(P)=T_{1}^{1}(P)$ by $%
G(A,B)=g_{it}g^{jl}A_{j}^{i}B_{l}^{t}$ for all $A,$ $B\in \Im _{1}^{1}\left(
P\right) $. The Cheeger-Gromoll type metric $^{CG}g$ is defined on $%
T_{1}^{1}(M)$ by the following three equations:%
\begin{equation}
{}^{CG}g\left( {}^{H}X,{}^{H}Y\right) =\text{ }^{V}({}g\left( X,Y\right) ),
\label{B3.1}
\end{equation}%
\begin{equation}
{}^{CG}g\left( {}^{V}A,{}^{H}Y\right) =0,  \label{B3.2}
\end{equation}%
\begin{equation}
^{CG}g\left( {}^{V}A,{}^{V}B\right) =\frac{1}{\alpha }{}%
^{V}(G(A,B)+G(A,t)G(B,t))  \label{B3.3}
\end{equation}%
for any $X,Y\in \Im _{0}^{1}\left( M\right) $ and $A,B\in \Im _{1}^{1}\left(
M\right) $, where $r^{2}=G\left( t,t\right) =g_{it}g^{jl}t_{j}^{i}t_{l}^{t}$
\ and $\alpha =1+r^{2}$. For the Levi-Civita connection of the
Cheeger-Gromoll type metric $^{CG}g$ we give the next theorem.

\begin{theorem}
\label{propo1} Let $(M,g)$ be a Riemannian manifold\textit{\ and }${}%
\widetilde{\nabla }$\textit{\ be the Levi-Civita connection of the tensor
bundle }$T_{1}^{1}(M)$\textit{\ equipped with the Cheeger-Gromoll type
metric ${}$}$^{CG}g$\textit{. }Then the corresponding Levi-Civita connection
satisfies the following relations: 
\begin{equation*}
\begin{array}{l}
{i)\mathrm{\;\;}{}}\widetilde{{\nabla }}{_{{}^{H}X}{}^{H}Y={}^{H}\left(
\nabla _{X}Y\right) +{\frac{1}{2}}(\tilde{\gamma}-\gamma )R(X,Y),} \\ 
{ii)\mathrm{\;\;}{}}\widetilde{{\nabla }}{_{{}^{H}X}{}^{V}B={}{\frac{1}{%
2\alpha }}{}^{H}\left( g^{bj}\,R(t_{b},B_{j})X+g_{ai}\,(t^{a}(g^{-1}\circ
R(\quad ,X)\tilde{B}^{\,i}\right) +{}^{V}\left( \nabla _{X}B\right) ,} \\ 
{iii)\mathrm{\;\;}}\widetilde{{\nabla }}{_{{}^{V}A}{}^{H}Y={\frac{1}{2\alpha 
}}{}^{H}\left( g^{bl}\,R(t_{b},A_{l})Y+g_{at}(t^{a}\,(g^{-1}\circ R(\quad ,Y)%
\widetilde{A}^{\,t})\right) ,} \\ 
{iv)\mathrm{\;\;\;}{}}\widetilde{{\nabla }}{_{{}^{V}A}{}^{V}B=-\frac{1}{%
\alpha }\left( {}^{CG}g\left( {}^{V}A,^{V}t\right) {}^{V}B+{}^{CG}g\left(
{}^{V}B,^{V}t\right) {}^{V}A\right) } \\ 
{+\frac{\alpha +1}{\alpha }{}^{CG}g\left( {}^{V}A,{}^{V}B\right) }^{V}t{-%
\frac{1}{\alpha }{}^{CG}g\left( {}^{V}A,^{V}t\right) {}^{CG}g\left(
{}^{V}B,^{V}t\right) }^{V}t%
\end{array}%
\end{equation*}%
\noindent for all $X,\;Y\in \Im _{0}^{1}(M)$ and $A,\;B\in \Im _{1}^{1}(M)$,
where $A_{l}=(A_{l}^{\,\,i})$, $\widetilde{A}^{t}=(g^{bl}A_{l}^{\quad
t})=(A_{\,.}^{b\,t})$, $t_{l}=(t_{l}^{\;a})$, $t^{a}=(t_{b}^{\;a})$, $%
R(\;\;,X)Y\in \Im _{1}^{1}(M),$ $g^{-1}\circ R(\;\;,X)Y\in \Im _{0}^{1}(M)$
(see, also \cite{Peyghan}).\bigskip
\end{theorem}

\begin{proof}
The connection \bigskip ${}\widetilde{{\nabla }}$ is characterized by the
Koszul formula:%
\begin{eqnarray*}
2^{CG}g({}\widetilde{{\nabla }}_{\widetilde{X}}\widetilde{Y},\widetilde{Z})
&=&\widetilde{X}(^{CG}g(\widetilde{Y},\widetilde{Z}))+\widetilde{Y}(^{CG}g(%
\widetilde{Z},\widetilde{X}))-\widetilde{Z}(^{CG}g(\widetilde{X},\widetilde{Y%
})) \\
-^{CG}g(\widetilde{X},[\widetilde{Y},\widetilde{Z}]) &+&^{CG}g(\widetilde{Y}%
,[\widetilde{Z},\widetilde{X}])+\text{ }^{CG}g(\widetilde{Z},[\widetilde{X},%
\widetilde{Y}])
\end{eqnarray*}%
for all vector fields $\widetilde{X},\widetilde{Y}$ and $\widetilde{Z}$ on $%
T_{1}^{1}(M)$. One can verify the Koszul formula for pairs $\widetilde{X}=$ $%
^{H}X,$ $^{V}A$ and $\widetilde{Y}=$ $^{H}Y,$ $^{V}A$ and $\widetilde{Z}=$ $%
^{H}Z,$ $^{V}C$. In calculations, the formulas (\ref{A2.3}), definition of
Cheeger-Gromoll type metric and the Bianchi identities for $R$ should be
applied. We omit standart calculations.
\end{proof}

The diagonal lift $^{D}\gamma $ of $\gamma \in \Im _{1}^{1}(M)$ to $%
T_{1}^{1}(M)$ is defined by%
\begin{eqnarray*}
^{D}\gamma ^{H}X &=&^{H}(\gamma (X)) \\
^{D}\gamma ^{V}A &=&-^{V}(A\circ \gamma )
\end{eqnarray*}%
for any $X\in \Im _{0}^{1}\left( M\right) $ and $A\in \Im _{1}^{1}\left(
M\right) $, where $A\circ \gamma =C(A\otimes \gamma )=A_{j}^{m}t_{m}^{i}$ 
\cite{Gezer2}. The diagonal lift $^{D}I$ of the identity tensor field $I\in
\Im _{1}^{1}(M)$ has the following properties:%
\begin{eqnarray*}
^{D}I^{H}X &=&^{H}X \\
^{D}I^{V}A &=&-^{V}A
\end{eqnarray*}%
and satisfies $(^{D}I)^{2}=I_{T_{1}^{1}(M)}$. Thus, $^{D}I$ is an almost
product structure on $T_{1}^{1}(M)$. We compute%
\begin{equation*}
P(\widetilde{X},\widetilde{Y})=^{CG}g(^{D}I\widetilde{X},\widetilde{Y}%
)-^{CG}g(\widetilde{X},^{D}I\widetilde{Y})
\end{equation*}%
for all $\widetilde{X},\widetilde{Y}$ $\in \Im _{0}^{1}\left(
T_{1}^{1}(M)\right) $. For pairs $\widetilde{X}=^{H}X,^{V}A$ and $\widetilde{%
Y}=^{H}Y,^{V}B$, by virtue of (\ref{B3.1})-(\ref{B3.3}), we get $P(%
\widetilde{X},\widetilde{Y})=0$, i.e. the Cheeger-Gromoll type metric $%
^{CG}g $ is pure with respect to the almost product structure $^{D}I$. Hence
we state the following theorem.

\begin{theorem}
\label{Theo1}Let $(M,g)$ be a Riemannian manifold and $T_{1}^{1}(M)$ be its $%
(1,1)-$tensor bundle equipped with the Cheeger-Gromoll type metric $^{CG}g$
and the almost product structure $^{D}I$. The triple $%
(T_{1}^{1}(M),^{D}I,^{CG}g)$ is a Riemannian almost product manifold.
\end{theorem}

We now give conditions for the Cheeger-Gromoll type metric $^{CG}g$ to be
decomposable with respect to the almost product structure $^{D}I$. We
calculate 
\begin{eqnarray*}
(\Phi _{^{D}I}{}^{CG}g)(\tilde{X},\tilde{Y},\tilde{Z}) &=&(^{D}I\tilde{X}%
)({}^{CG}g(\tilde{Y},\tilde{Z}))-\tilde{X}(^{CG}g(^{D}I\tilde{Y},\tilde{Z}))
\\
&+&{}^{CG}g((L_{\tilde{Y}}\text{ }^{D}I)\tilde{X},\tilde{Z})+{}^{CG}g_{f}(%
\tilde{Y},(L_{\tilde{Z}}\text{ }^{D}I)\tilde{X})
\end{eqnarray*}%
for all $\tilde{X},\tilde{Y},\tilde{Z}\in \Im _{0}^{1}(T_{1}^{1}(M))$. For
pairs $\tilde{X}=^{H}X,^{V}A$, $\widetilde{Y}=^{H}Y,^{V}B$ and $\widetilde{Z}%
=^{H}Z,{}^{V}C$, then we get%
\begin{eqnarray}
(\Phi _{^{D}I}{}^{CG}g)({}^{H}X,{}^{V}B,{}^{H}Z) &=&2^{CG}g({}^{V}B,{(\tilde{%
\gamma}-\gamma )}R(Z,X)),  \label{B3.4} \\
(\Phi _{^{D}I}{}^{CG}g)({}^{H}X,{}^{H}Y,{}^{V}C) &=&2^{CG}g({(\tilde{\gamma}%
-\gamma )R(Y},{X}),^{V}C),  \notag \\
Otherwise &=&0.  \notag
\end{eqnarray}%
Therefore, we have the following result.

\begin{theorem}
\label{Theo2}Let $(M,g)$ be a Riemannian manifold and let $T_{1}^{1}(M)$ be
its $(1,1)-$tensor bundle equipped with the Cheeger-Gromoll type metric $%
^{CG}g$ and the almost product structure $^{D}I$. The triple $\left(
T_{1}^{1}(M),^{D}I,{}^{CG}g\right) $ is a locally decomposable Riemannian
manifold if and only if $M$ is flat.
\end{theorem}

\begin{remark}
Let $(M,g)$ be a flat Riemannian manifold. In the case the $(1,1)-$tensor
bundle $T_{1}^{1}(M)$ equipped with the Cheeger-Gromoll type metric $^{CG}g$
over the flat Riemannian manifold $(M,g)$ is unflat (see \cite{Peyghan}).
\end{remark}

Let $(M,\varphi ,g)$ be a non-integrable almost product manifold with a pure
metric. A Riemannian almost product manifold $(M,\varphi ,g)$ is a
Riemannian almost product $\mathcal{W}_{3}-$manifold if $\underset{X,Y,Z}{%
\sigma }g((\nabla _{X}\varphi )Y,Z)=0$, where $\sigma $ is the cyclic sum by 
$X,Y,Z$ \cite{Staikova}. In \cite{Salimov4}, the authors proved that $%
\underset{X,Y,Z}{\sigma }g((\nabla _{X}\varphi )Y,Z)=0$ is equivalent to $%
(\Phi _{\varphi }g)(X,Y,Z)+(\Phi _{\varphi }g)(Y,Z,X)+(\Phi _{\varphi
}g)(Z,X,Y)=0.$ We compute%
\begin{equation*}
A(\tilde{X},\tilde{Y},\tilde{Z})=(\Phi _{^{D}I}{}^{CG}g)(\tilde{X},\tilde{Y},%
\tilde{Z})+(\Phi _{^{D}I}{}^{CG}g)(\widetilde{Y},\widetilde{Z},\widetilde{X}%
)+(\Phi _{^{D}I}{}^{CG}g)(\widetilde{Z},\widetilde{X},\widetilde{Y})
\end{equation*}%
for all $\widetilde{X},\widetilde{Y},\widetilde{Z}\in \Im
_{0}^{1}(T_{1}^{1}(M)).$ By means of (\ref{B3.4}), we have $A(\tilde{X},%
\tilde{Y},\tilde{Z})=0$ for all $\widetilde{X},\widetilde{Y},\widetilde{Z}%
\in \Im _{0}^{1}(T_{1}^{1}(M)).$ Hence we state the following theorem.

\begin{theorem}
\label{Theo5}Let $(M,g)$ be a Riemannian manifold and $T_{1}^{1}(M)$ be its $%
(1,1)-$tensor bundle equipped with the Cheeger-Gromoll type metric $^{CG}g$
and the almost product structure $^{D}I$. The triple $%
(T_{1}^{1}(M),^{D}I,^{CG}g)$ is a Riemannian almost product $\mathcal{W}%
_{3}- $manifold.
\end{theorem}

\begin{remark}
Let $(M,g)$ be a Riemannian manifold and let $T_{1}^{1}(M)$ be its $(1,1)-$%
tensor bundle equipped with the Cheeger-Gromoll type metric $^{CG}g.$
Another almost product structure on $T_{1}^{1}(M)$ is defined by the formulas%
\begin{eqnarray*}
J(^{H}X) &=&-^{H}X \\
J({}^{V}A) &=&^{V}A
\end{eqnarray*}%
for any $X\in \Im _{0}^{1}\left( M\right) $ and $A\in \Im _{1}^{1}\left(
M\right) $. The Cheeger-Gromoll type metric $^{CG}g$ is pure with respect to 
$J$, e.i. the triple $\left( T_{1}^{1}(M),J,{}^{CG}g\right) $ is a
Riemannian almost product manifold. Also, by using $\Phi -$operator, we can
say that the Cheeger-Gromoll type metric $^{CG}g$ is decomposable with
respect to $J$ if and only if $M$ is flat. Finally the triple $%
(T_{1}^{1}(M),J,^{CG}g)$ is another Riemannian almost product $\mathcal{W}%
_{3}-$manifold.
\end{remark}

O. Gil-Medrano and A.M. Naveira proved that both distributions of the almost
product structure on the Riemannian almost product manifold ($M,F,g)$ are
totally geodesic if and only if $\underset{X,Y,Z}{\sigma }g((\nabla
_{X}F)Y,Z)=0$ for any $X,Y,Z\in \Im _{0}^{1}(M)$ \cite{Gil}$.$ As a
consequence of Theorem \ref{Theo5}, we have the following.

\begin{corollary}
Both distributions of the Riemannian almost product manifold $(T_{1}^{1}(M),$
$^{D}I,^{CG}g)$ are totally geodesic.
\end{corollary}

\section{Product conjugate connections on the $(1,1)-$tensor bundle with
Cheeger-Gromoll type metric}

Let $F$ be an almost product structure and $\nabla $ be a linear connection
on an $n$-dimensional Riemannian manifold $M.$ The product conjugate
connection $\nabla ^{(F)}$ of $\nabla $ is defined by%
\begin{equation*}
\nabla _{X}^{(F)}Y=F(\nabla _{X}FY)
\end{equation*}%
for all $X,Y\in \Im _{0}^{1}(M).$ If $(M,F,g)$ is a Riemannian almost
product manifold, then $(\nabla _{X}^{(F)}g)(FY,FZ)=(\nabla _{X}g)(Y,Z)$,
i.e. $\nabla $ is a metric connection with respect to $g$ if and only if $%
\nabla ^{(F)}$ is so. From this, we can say that if $\nabla $ is the
Levi-Civita connection of $g$, then $\nabla ^{(F)}$ is a metric connection
with respect to $g$ \cite{Blaga}.

By the almost product structure $^{D}I$ and the Levi-Civita connection $%
\widetilde{\nabla }$ given by Theorem \ref{propo1}, we write the product
conjugate connection $\widetilde{\nabla }^{(^{D}I)}$ of $\widetilde{\nabla }$
as follows:%
\begin{equation*}
\widetilde{\nabla }_{\widetilde{X}}^{(^{D}I)\text{ \ }}\widetilde{Y}=\text{ }%
^{D}I(\widetilde{\nabla }_{\widetilde{X}}\text{ }^{D}I\widetilde{Y})
\end{equation*}%
for all $\widetilde{X},\widetilde{Y}\in \Im _{0}^{1}(T_{1}^{1}(M)).$ Also
note that $\widetilde{\nabla }^{(^{D}I)}$ is a metric connection of the
Cheeger-Gromoll type metric $^{CG}g$. The standart calculations give the
following theorem.

\begin{theorem}
Let $(M,g)$ be a Riemannian manifold and let $T_{1}^{1}(M)$ be its $(1,1)-$%
tensor bundle equipped with the Cheeger-Gromoll type metric $^{CG}g$ and the
almost product structure $^{D}I$. Then the product conjugate connection (or
metric connection) $\widetilde{\nabla }^{(^{D}I)}$ satisfies 
\begin{equation*}
\begin{array}{l}
{i)\mathrm{\;\;}{}}\widetilde{{\nabla }}^{(^{D}I)}{_{^{H}X}{}^{H}Y={}^{H}%
\left( \nabla _{X}Y\right) -{\frac{1}{2}}(\tilde{\gamma}-\gamma )R(X,Y),} \\ 
{ii)\mathrm{\;}\widetilde{{\nabla }}}^{(^{D}I)}{{_{{}^{H}X}{}^{V}B=-{}{\frac{%
1}{2\alpha }}{}^{H}}}\left(
g^{bj}\,R(t_{b},B_{j})X+g_{ai}\,(t^{a}(g^{-1}\circ R(\quad ,X)\tilde{B}%
^{\,i})\right) {{+{}^{V}\left( \nabla _{X}B\right) },} \\ 
{iii)\mathrm{\;\;}\widetilde{{\nabla }}}^{(^{D}I)}{{_{{}^{V}A}{}^{H}Y={\frac{%
1}{2\alpha }}{}^{H}\left( g^{bl}\,R(t_{b},A_{l})Y+g_{at}(t^{a}\,(g^{-1}\circ
R(\quad ,Y)\widetilde{A}^{\,t}))\right) },} \\ 
{iv)\mathrm{\;\;\;}{}}\widetilde{{\nabla }}^{(^{D}I)}{_{{}^{V}A}{}^{V}B=-%
\frac{1}{\alpha }\left( {}^{CG}g\left( {}^{V}A,^{V}t\right)
{}^{V}B+{}^{CG}g\left( {}^{V}B,^{V}t\right) {}^{V}A\right) } \\ 
{+\frac{\alpha +1}{\alpha }{}^{CG}g\left( {}^{V}A,{}^{V}B\right) }^{V}t{-%
\frac{1}{\alpha }{}^{CG}g\left( {}^{V}A,^{V}t\right) {}^{CG}g\left(
{}^{V}B,^{V}t\right) }^{V}t{.}%
\end{array}%
\end{equation*}
\end{theorem}

\begin{remark}
Let $(M,g)$ be a flat Riemannian manifold and let $T_{1}^{1}(M)$ be its $%
(1,1)-$tensor bundle equipped with the Cheeger-Gromoll type metric $^{CG}g$
and the almost product structure $^{D}I.$ Then the Levi-Civita connection $%
\widetilde{\nabla }$ of $^{CG}g$ coincides with the product conjugate
connection (or metric connection) $\widetilde{\nabla }^{(^{D}I)}$
constructed by the Levi-Civita connection $\widetilde{\nabla }$ of $^{CG}g$
and the almost product structure $^{D}I$.
\end{remark}

The relationship between curvature tensors $R_{\nabla }$ and $R_{\nabla
^{(F)}}$ of the connections $\nabla $ and $\nabla ^{(F)}$ is as follows: $%
R_{\nabla ^{(F)}}(X,Y,Z)=F(R_{\nabla }(X,Y,FZ)$ for all $X,Y,Z\in \Im
_{0}^{1}(M)$ \cite{Blaga}$.$ Using the almost product structure $^{D}I$ and
curvature tensor of the Cheeger-Gromoll type metric $^{CG}g$ (for curvature
tensor of $^{CG}g,$ see \cite{Peyghan}), by means of $\widetilde{R}_{%
\widetilde{\nabla }^{(^{D}I)}}(\widetilde{X},\widetilde{Y},\widetilde{Z}%
)=^{D}I(\widetilde{R}_{\widetilde{\nabla }}(\widetilde{X},\widetilde{Y},^{D}I%
\widetilde{Z})$, the curvature tensor $\widetilde{R}_{\widetilde{\nabla }%
^{(^{D}I)}}$ of the product conjugate connection (or metric connection) $%
\widetilde{\nabla }^{(^{D}I)}$ can be easily written. Also note that another
metric connection of the Cheeger-Gromoll type metric $^{CG}g$ can be
constructed by using the almost product structure $J$.

The torsion tensor $^{\widetilde{\nabla }^{(^{D}I)}}\widetilde{T}$ of the
product conjugate connection $\widetilde{\nabla }^{(^{D}I)}$ (or metric
connection of $^{CG}g$) has the following properties:%
\begin{eqnarray*}
^{\widetilde{\nabla }^{(^{D}I)}}\widetilde{T}(^{H}X,^{H}Y) &=&-2{(\tilde{%
\gamma}-\gamma )R(X,Y)} \\
^{\widetilde{\nabla }^{(^{D}I)}}\widetilde{T}(^{H}X,^{V}B) &=&{{-{}{\frac{1}{%
\alpha }}{}^{H}}}\left( g^{bj}\,R(t_{b},B_{j})X+g_{ai}\,(t^{a}(g^{-1}\circ
R(\quad ,X)\tilde{B}^{\,i})\right) \\
^{\widetilde{\nabla }^{(^{D}I)}}\widetilde{T}(^{V}A,^{H}Y) &=&{{{\frac{1}{%
\alpha }}{}^{H}}}\left( g^{bj}\,R(t_{b},A_{j})Y+g_{ai}\,(t^{a}(g^{-1}\circ
R(\quad ,Y)\widetilde{A}^{\,i})\right) \\
^{\widetilde{\nabla }^{(^{D}I)}}\widetilde{T}(^{V}A,^{V}B) &=&0.
\end{eqnarray*}

The last equations lead to the following result.

\begin{theorem}
Let $(M,g)$ be a Riemannian manifold and let $T_{1}^{1}(M)$ be its $(1,1)-$%
tensor bundle equipped with the Cheeger-Gromoll type metric $^{CG}g$ and the
almost product structure $^{D}I$. The product conjugate connection (or
metric connection) $\widetilde{\nabla }^{(^{D}I)}$ constructed by the
Levi-Civita connection $\widetilde{\nabla }$ of $^{CG}g$ and the almost
product structure $^{D}I$ is symmetric if and only if $M$ is flat.
\end{theorem}

\end{document}